\documentclass[12pt,reqno]{amsart}
\usepackage{graphics, color, amsmath, amsfonts, amssymb, amsthm, amscd, fullpage,verbatim,graphicx,mathrsfs, paralist}
\usepackage{hyperref}
\theoremstyle{plain}
\newtheorem{theorem}{Theorem}[section]
\newtheorem{lemma}[theorem]{Lemma}

\newtheorem{corollary}[theorem]{Corollary}

\theoremstyle{definition}
\newtheorem{definition}{Definition}[section]
\newtheorem{remark}[definition]{Remark}
\newcommand{\N}{\mathbb N}

\newcommand{\e}{\varepsilon}

\newcommand{\eqdef}{\stackrel{\mathrm{def}}{=}}
\renewcommand{\d}{\delta}
\renewcommand{\le}{\leqslant}
\renewcommand{\ge}{\geqslant}

\newcommand{\U}{\mathscr U}
\newcommand{\R}{\mathbb R}
\newcommand{\vol}{\mathrm{vol}}
\newcommand{\n}{\mathcal N}
\newcommand{\f}{\phi}
\renewcommand{\setminus}{\smallsetminus}
\begin{document}

\title{Bourgain's discretization Theorem}\thanks{O.~G. was supported by Fondation Sciences Math\'ematiques de Paris and NSF grant CCF-0832795. A.~N.
was supported by NSF grant CCF-0832795, BSF grant
2006009 and the Packard Foundation. G.~S. was supported by the Israel Science Foundation. This work was completed when A.~N. and G.~S. were visiting the Quantitative Geometry program at MSRI}

\author{Ohad Giladi}
\address{Institut de Math\'ematiques de Jussieu, Universit\'e Paris VI}
\email{giladi@math.jussieu.fr}

\author{Assaf Naor}
\address{Courant Institute, New York University}
\email{naor@cims.nyu.edu}

\author{Gideon Schechtman}
\address{Department of Mathematics, Weizmann Institute of Science}
\email{gideon@weizmann.ac.il}

\maketitle

\begin{abstract}
Bourgain's discretization theorem asserts that there exists a universal constant $C\in (0,\infty)$ with the following property. Let $X,Y$ be Banach spaces with $\dim X=n$. Fix $D\in (1,\infty)$ and set $\d= e^{-n^{Cn}}$. Assume that $\mathcal N$ is a $\d$-net in the unit ball of $X$ and that $\mathcal N$ admits a bi-Lipschitz embedding into $Y$ with distortion at most $D$. Then the entire space $X$ admits a bi-Lipschitz embedding into $Y$ with distortion at most $CD$. This mostly expository article is devoted to a detailed presentation of a proof of Bourgain's theorem.

We also obtain an improvement of Bourgain's theorem in the important case when $Y=L_p$ for some $p\in [1,\infty)$: in this case it suffices to take $\delta= C^{-1}n^{-5/2}$ for the same conclusion to hold true. The case $p=1$ of this improved discretization result has the following consequence. For arbitrarily large $n\in \N$ there exists a family $\mathscr Y$ of $n$-point subsets of $\{1,\ldots,n\}^2\subseteq \R^2$ such that if we write $|\mathscr Y|= N$ then any $L_1$ embedding of $\mathscr Y$, equipped with the  Earthmover metric (a.k.a. transportation cost metric or minimumum weight matching metric)  incurs distortion at least a constant multiple of $\sqrt{\log\log N}$; the previously best known lower bound for this problem was a constant multiple of $\sqrt{\log\log \log N}$.
\end{abstract}

\section{Introduction}

If $(X,d_X)$ and $(Y,d_Y)$ are metric spaces then the (bi-Lipschitz) distortion of $X$ in $Y$, denoted $c_Y(X)$, is the infimum over those $D\in [1,\infty]$ such that there exists $f:X\to Y$ and $s\in (0,\infty)$ satisfying $sd_X(x,y)\le d_Y(f(x),f(y))\le Dsd_X(x,y)$ for all $x,y\in X$. Assume now that $X,Y$ are Banach spaces, with unit balls $B_X,B_Y$, respectively. Assume furthermore that $X$ is finite dimensional. It then follows from general principles that for every $\e\in (0,1)$ there exists $\d\in (0,1)$ such that for every $\d$-net $\mathcal{N}_\d$ in $B_X$ (recall that a $\delta$-net is a maximal $\delta$-separated subset of $B_X$) we have $c_Y(\mathcal{N}_\d)\ge (1-\e) c_Y(X)$. Indeed, set $D=c_Y(X)$ and assume  that for all $k\in \N$ there is a $1/k$-net $\mathcal N_{1/k}$ of $B_X$ and  $f_k:\mathcal N_{1/k}\to Y$ satisfying $\|x-y\|_X\le \|f_k(x)-f_k(y)\|_Y\le (1-\e)D\|x-y\|_X$ for all $x,y\in \mathcal N_{1/k}$. For each $x\in B_X$ fix some $z_k(x)\in \mathcal N_{1/k}$ satisfying $\|x-z_k(x)\|_X\le 1/k$. Let $\mathscr U$ be a free ultrafilter on $\N$. Consider the ultrapower $Y_\U$, i.e., the space of equivalence classes of bounded $Y$-valued sequences  modulo the equivalence relation $(x_k)_{k=1}^\infty\sim (y_k)_{k=1}^\infty\iff \lim_{k\to \mathscr U}\|x_k-y_k\|_Y=0$, equipped with the norm $\|(x_k)_{k=1}^\infty/\sim\|_{Y_\U}=\lim_{k\to \U} \|x_k\|_Y$. Define $f_\U:B_X\to Y_\U$ by $f_\U(x)=(f_k(z_k(x))_{k=1}^\infty)/\sim$. Then $\|x-y\|_X\le \|f_\U(x)-f_\U(y)\|_{Y_\U}\le (1-\e)D\|x-y\|_X$ for all $x,y\in X$. By a (nontrivial) $w^*$-G\^ateaux differentiability argument due to Heinrich and Mankiewicz~\cite{HM82} it now follows that there exists a linear mapping $T_1:X\to (Y_\U)^{**}$ satisfying $\|x\|_X\le \|T_1x\|_{(Y_\U)^{**}}\le (1-\e/2)D\|x\|_X$ for all $x\in X$. Since $X$, and hence also $T_1X$, is finite dimensional, the Principle of Local Reflexivity~\cite{LR69} says there exists a linear mapping $T_2: T_1X\to Y_\U$ satisfying $\|y\|_{(Y_\U)^{**}}\le \|T_2y\|_{Y_\U}\le (1+\e/5)\|y\|_{(Y_\U)^{**}}$ for all $y\in T_1X$. By general properties of ultrapowers (see~\cite{Hei80}) there exists a linear mapping $T_3:T_2T_1X\to Y$ satisfying $\|y\|_{Y_\U}\le \|T_3 y\|_Y\le (1+\e/5)\|y\|_{Y_\U}$ for all $y\in T_2T_1X$. By considering $T_3T_2T_1:X\to Y$ we have $D=c_Y(X)\le (1-\e/2)(1+\e/5)^2D$, a contradiction.

The argument sketched above is due to Heinrich and Mankiewicz~\cite{HM82}. An earlier and different argument establishing the existence of $\d$ is due to important work of Ribe~\cite{Ribe76}.  See the book~\cite{BL00} for a detailed exposition of both arguments. These proofs do not give a concrete estimate on $\d$. The first purpose of the present article, which is mainly expository, is to present in detail a different approach due to Bourgain~\cite{Bou87} which does yield an estimate on $\delta$. Before stating Bourgain's theorem, it will be convenient to introduce the following quantity.


\begin{definition}[Discretization modulus] For $\e\in (0,1)$ let $\d_{X\hookrightarrow Y}(\e)$ be the supremum over those $\d\in (0,1)$ such that every $\d$-net $\mathcal N_\d$ in $B_X$ satisfies $c_Y(\mathcal{N}_\d)\ge (1-\e) c_Y(X)$.
\end{definition}

\begin{theorem}[Bourgain's discretization theorem]\label{thm:bourgain intro} There exists $C\in (0,\infty)$ such that for every two Banach spaces $X,Y$ with $\dim X=n<\infty$ and $\dim Y=\infty$, and every $\e\in (0,1)$, we have
\begin{equation}\label{eq:bourgain bound}
\d_{X\hookrightarrow Y}(\e)\ge e^{-(n/\e)^{Cn}}.
\end{equation}
\end{theorem}

Theorem~\ref{thm:bourgain intro} was proved by Bourgain in~\cite{Bou87} for some fixed $\e_0\in (0,1)$. The above statement requires small technical modifications of Bourgain's argument, but these are minor and all the conceptual ideas presented in the proof of Theorem~\ref{thm:bourgain intro} below are due to Bourgain.  Readers might notice that our presentation of the proof of Theorem~\ref{thm:bourgain intro} seems somewhat different from~\cite{Bou87}, but this impression is superficial; the exposition below is merely a restructuring of Bourgain's argument.

We note that it is possible to refine the estimate~\eqref{eq:bourgain bound} so as to depend on the distortion $c_Y(X)$. Specifically, we have the bound
\begin{equation}\label{eq:bourgain bound D}
\d_{X\hookrightarrow Y}(\e)\ge e^{-(c_Y(X)/\e)^{Cn}}.
\end{equation}
The estimate~\eqref{eq:bourgain bound D} implies~\eqref{eq:bourgain bound} since due to Dvoretzky's theorem~\cite{Dvo60}  $c_Y(\ell_2^n)=1$, and therefore $c_Y(X)\le \sqrt{n}$ by John's theorem~\cite{Jo}. If we do not assume that $\dim Y=\infty$ then we necessarily have $\dim Y\ge n$ since otherwise $c_Y(X)=\infty$, making~\eqref{eq:bourgain bound D} hold vacuously. Thus, by John's theorem once more, $c_Y(X)\le n$, and again we see that~\eqref{eq:bourgain bound D} implies~\eqref{eq:bourgain bound}. The proof below will establish~\eqref{eq:bourgain bound D}, and not only the slightly weaker statement~\eqref{eq:bourgain bound}. We remark that Bourgain's discretization theorem is often quoted with the conclusion that if $\delta$ is at most as large as the right hand side of~\eqref{eq:bourgain bound D} and $\mathcal N_\d$ is a $\d$-net  of $B_X$ then $Y$ admits a {\em linear} embedding into $Y$ whose distortion is at most $c_Y(\mathcal N_\delta)/(1-\e)$. The Heinrich-Mankiewicz argument described above shows that for finite dimensional spaces $X$, a bound on $c_Y(X)$ immediately implies the same bound when the bi-Lipschitz embedding is required to be linear. For this reason we ignore the distinction between linear and nonlinear bi-Lipschitz embeddings, noting also that for certain applications (e.g., in computer science), one does not need to know that embeddings are linear.


We do not know how close is the estimate~\eqref{eq:bourgain bound} to being asymptotically optimal, though we conjecture that it can be improved. The issue of finding examples showing that $\d_{X\hookrightarrow Y}(\e)$ must be small has not been sufficiently investigated in the literature. The known upper bounds on $\d_{X\hookrightarrow Y}(\e)$ are very far from~\eqref{eq:bourgain bound}. For example, the metric space $(\ell_1^n,\sqrt{\|x-y\|_1})$ embeds isometrically into $L_2$ (see~\cite{DL97}). It follows that any $\d$-net in $B_{\ell_1^n}$ embeds into $L_2$ with distortion at most $\sqrt{2/\d}$. Contrasting this with $c_{L_2}(\ell_1^n)=\sqrt{n}$ shows  that $\delta_{\ell_1^n\hookrightarrow L_2}(\e)\le 2/\left((1-\e)^2n\right)$.

It turns out that a method that was introduced by Johnson, Maurey and Schechtman~\cite{JMS09} (for a different purpose) can be used to obtain improved bounds on $\delta_{X\hookrightarrow Y}(\e)$ for certain Banach spaces $Y$, including all $L_p$ spaces, $p\in [1,\infty)$; the second purpose of this article is to present this result. To state our result recall that if $(\Omega,\nu)$ is a measure space and $(Z,\|\cdot\|_X)$ is a Banach space then for $p\in [1,\infty]$ the vector valued $L_p$ space $L_p(\nu,Z)$ is the space of all equivalence classes of measurable functions $f:\Omega\to Z$ such that $\|f\|_{L_p(\nu,Z)}^p=\int_\Omega\|f\|_Z^pd\nu<\infty$ (and $\|f\|_{L_\infty(\nu,Z)}=\mathrm{esssup}_{\omega\in \Omega}\|f(\omega)\|_Y$).

\begin{theorem}~\label{thm:with Gid}
There exists a universal constant $\kappa\in (0,\infty)$ with the following property. Assume that $\d,\e\in (0,1)$ and $D\in [1,\infty)$ satisfy $\d\le \kappa\e^2/(n^2D)$. Let $X,Y$ be Banach spaces with $\dim X=n<\infty$, and let $\mathcal N_\d$ be a $\d$-net in $B_X$. Assume that $c_Y(\mathcal N_\d)\le D$. Then there exists a separable probability space $(\Omega,\nu)$, a finite dimensional linear subspace $Z\subseteq Y$, and a linear operator $T:X\to L_\infty(\nu,Z)$ satisfying
$$
\forall x\in X,\quad \frac{1-\e}{D}\|x\|_X\le \|Tx\|_{L_1(\nu,Z)}\le  \|Tx\|_{L_\infty(\nu,Z)}\le (1+\e)\|x\|_X.
$$
\end{theorem}
Theorem~\ref{thm:with Gid} is proved in Section~\ref{sec:use JMS}; as we mentioned above, its proof builds heavily on ideas from~\cite{JMS09}. Because $\nu$ is a probability measure, for all $p\in [1,\infty]$ and all $h\in L_\infty(\nu,Y)$ we have $\|h\|_{L_1(\nu,Y)}\le \|h\|_{L_p(\nu,Y)}\le \|h\|_{L_\infty(\nu,Y)}$. Therefore, the following statement is a consequence of Theorem~\ref{thm:with Gid}.
\begin{equation}\label{eq:L_p(Z)}
\d\le \frac{\kappa\e^2}{n^2c_Y(\mathcal N_\d)}\implies \forall p\in [1,\infty),\quad c_Y(\mathcal N_\d)\ge \frac{1-\e}{1+\e}c_{L_p(\nu,Y)}(X).
\end{equation}
We explained above that if $Y$ is infinite dimensional then  $c_Y(\mathcal N_\d)\le \sqrt{n}$. It therefore follows from~\eqref{eq:L_p(Z)} that if $L_p(\nu,Y)$ admits an isometric embedding into $Y$, as is the case when $Y=L_p$, then $\d_{X\hookrightarrow Y}(\e)\ge \kappa\e^2/(n^{5/2})$. This is recorded for future reference as the following corollary.

\begin{corollary}\label{coro:5/2}
There exists a universal constant $\kappa\in (0,\infty)$ such that for every $p\in [1,\infty)$ and $\e\in (0,1)$, for every $n$-dimensional Banach space $X$ we have
\begin{equation}\label{eq:L_p delta}
\delta_{X\hookrightarrow L_p}(\e)\ge \frac{\kappa \e^2}{n^{5/2}}.
\end{equation}
\end{corollary}
There is a direct application of the case $p=1$ of Corollary~\ref{coro:5/2} to the minimum cost matching metric on $\R^2$. Given $n\in \N$, consider the following metric $\tau$ on the set of all $n$-point subsets  of $\R^2$, known as the minimum cost matching metric.
$$
\tau(A,B) =\min\left\{\sum_{a\in A} \|a-f(a)\|_2:\ f:A\to B\ \mathrm{is\ a\ bijection}\right\}.
$$
\begin{corollary}\label{cor:EMD} There exists a universal constant $c\in (0,\infty)$ with the following property. For arbitrarily large $n\in \N$ there exists a family $\mathscr Y$ of $n$-point subsets of $\{1,\ldots,n\}^2\subseteq \R^2$ such that if we write $|\mathscr Y|= N$ then $c_{L_1}(\mathscr Y,\tau)\ge c\sqrt{\log\log N}.$
\end{corollary}
The previously best known lower bound in the context of Corollary~\ref{cor:EMD}, due to~\cite{NS07}, was $c_{L_1}(\mathscr Y,\tau)\ge c\sqrt{\log\log\log N}$. We refer to~\cite{NS07} for an explanation of the relevance of such problems to theoretical computer science. The deduction of Corollary~\ref{cor:EMD} from Corollary~\ref{coro:5/2} follows mutatis mutandis from the argument in~\cite[Sec.~3.1]{NS07}, the only difference being the use of the estimate~\eqref{eq:L_p delta} when $p=1$ rather than the estimate~\eqref{eq:bourgain bound} when $Y=L_1$.

For an infinite dimensional Banach space $Y$ define
$$
\delta_n(Y)\eqdef\inf\left\{\d_{X\hookrightarrow Y}\left(1/2\right):\ X\ \mathrm{is\ an\ }n\ \mathrm{dimensional\ Banach\ space}\right\},
$$
and set
$$
\delta_n\eqdef \inf\left\{\d_n(Y):\ Y\ \mathrm{is\ an\ infinite\ dimensional\ Banach\ space}\right\}.
$$
Theorem~\ref{thm:bourgain intro} raises natural geometric questions. Specifically, what is the asymptotic behavior of $\delta_n$ as $n\to \infty$?  The difficulty of this question does not necessarily arise from the need to consider all, potentially ``exotic", Banach spaces $Y$. In fact, the above discussion shows that $\Omega(1/n^{5/2})\le \delta_n(L_2)\le O(1/n)$, so we ask explicitly what is the asymptotic behavior of $\delta_n(L_2)$ as $n\to \infty$? For applications to computer science (see e.g. \cite{NS07}) it is especially important to bound $\delta_n(L_1)$, so we also single out the problem of evaluating the asymptotic behavior of $\delta_n(L_1)$ as $n\to \infty$. Recently, two alternative proofs of Theorem~\ref{thm:bourgain intro} that work for certain special classes of spaces $Y$ were obtained in~\cite{LN11,HLN11}, using different techniques than those presented here (one based on a quantitative differentiation theorem, and the other on vector-valued Littlewood-Paley theory). These new proofs yield, however, the same bound as~\eqref{eq:bourgain bound}.  The proof of Theorem~\ref{thm:bourgain intro} presented below is the only known proof of Theorem~\ref{thm:bourgain intro} that works in full generality.

\begin{remark}
The questions presented above are part of a more general discretization problem in embedding theory. One often needs to prove nonembeddability results for finite spaces, where the distortion is related to their cardinality. In many cases it is, however, easier to prove nonembeddability results for infinite spaces, using techniques that are available for continuous objects. It is natural to then prove a discretization theorem, i.e., a statement that transfers a nonembeddability theorem from a continuous object to its finite nets, with control on their cardinality. This general scheme was used several times in the literature, especially in connection to applications of embedding theory to computer science; see for example~\cite{NS07}, where Bourgain's discretization theorem plays an explicit role, and also, in a different context, \cite{ckn}. The latter example deals with the Heisenberg group rather than Banach spaces, the discretization in question being of an infinitary nonembeddability theorem of Cheeger and Kleiner~\cite{CK10}. It would be of interest to study the analogue of Bourgain's discretization theorem in the context of Carnot groups. This can be viewed as asking for a quantitative version of a classical theorem of Pansu~\cite{Pan89}. In the special case of embeddings of the Heisenberg group into Hilbert space, a different approach was used in~\cite{ANT10} to obtain a sharp result of this type.
\end{remark}

\begin{remark}
A Banach space $Z$ is said to be finitely representable in a Banach space $Y$ if there exists $K\in [1,\infty)$ such that for every finite dimensional subspace $X\subseteq Z$ there exists an injective linear operator $T:X\to Y$ satisfying $\|T\|\cdot\|T^{-1}\|\le K$. A theorem of Ribe~\cite{Ribe76} states that if $Z$ and $Y$ are uniformly homeomorphic, i.e., there exists a homeomorphism $f:Z\to Y$ such that both $f$ and $f^{-1}$ are uniformly continuous, then $Z$ is finitely representable in $Y$ and vice versa. This  rigidity phenomenon suggests that isomorphic invariants of Banach spaces which are defined using statements about finitely many vectors are preserved under uniform homemorphisms, and as such one might hope to reformulate them in a way that is explicitly nonlinear, i.e., while only making use of the metric structure and without making any reference to the linear structure. Once this (usually nontrivial) task is achieved, one can hope to transfer some of the linear theory of Banach spaces to the context of general metric spaces. This so called ``Ribe program" was put forth by Bourgain in~\cite{Bou86}; a research program that attracted the work of many mathematicians in the past 25 years, and has had far reaching consequences in areas such as metric geometry, theoretical computer science, and group theory.  The argument that we presented for the positivity of $\delta_{X\hookrightarrow Y}(\e)$ implies Ribe's rigidity theorem. Indeed, it is a classical observation~\cite{CK63} that if $f:Z\to Y$ is a uniform homeomorphism then it is bi-Lipschitz for large distances, i.e., for every $d\in (0,\infty)$ there exists $L\in (0,\infty)$ such that $L^{-1}\|x-y\|_Z\le \|f(x)-f(y)\|_Y\le L\|x-y\|_Z$ whenever $x,y\in Z$ satisfy $\|x-y\|_Z\ge d$. Consequently, if $X\subseteq Z$ is a finite dimensional subspace then $d$-nets in $rB_X$ embed into $Y$ with distortion at most $L^2$ for every $r>d$. By rescaling, the same assertion holds for $\d$-nets in $B_X$ for every $\d\in (0,1)$. Hence $X$ admits a linear embedding into $Y$ with distortion is at most $2L^2$. For this reason, in~\cite{Bou87} Bourgain calls his discretization theorem a quantitative version of Ribe's finite representability theorem. Sufficiently good improved lower bounds on $\delta_{X\hookrightarrow Y}(\e)$ are expected to have impact on the Ribe program.
\end{remark}

\section{The strategy of the proof of Theorem~\ref{thm:bourgain intro}}\label{sec:strategy}

From now on $(X,\|\cdot\|_X)$ will be a fixed $n$-dimensional normed space ($n>1$), with unit ball $B_X=\{x\in X:\ \|x\|_X\le 1\}$ and unit sphere $S_X=\{x\in X:\ \|x\|_X=1\}$. We will identify $X$ with $\R^n$, and by John's theorem~\cite{Jo} we will assume without loss of generality that the standard Euclidean norm $\|\cdot\|_2$ on $\R^n$ satisfies
\begin{equation}\label{eq:john assumption}
\forall\ x\in X,\quad \frac{1}{\sqrt{n}}\|x\|_2\le \|x\|_X\le \|x\|_2.
\end{equation}

Fix $\e,\delta\in (0,1/8)$ and let $\n_\d$ be a fixed $\d$-net in $B_X$. We also fix $D\in (1,\infty)$, a Banach space $(Y,\|\cdot\|_Y)$, and a mapping $f: \n_\d\to Y$ satisfying
\begin{equation}\label{eq:assumption on net}
\forall\ x,y\in \n_\d,\quad \frac{1}{D}\|x-y\|_X\le\|f(x)-f(y)\|_Y\le \|x-y\|_X.
\end{equation}
By translating $f$, we assume without loss of generality that $f(\n_\d)\subseteq 2B_Y$.
Our goal will be to show that provided $\d$ is small enough, namely $\delta\le e^{-(D/\e)^{Cn}}$, there exists an injective linear operator $T:X\to Y$ satisfying $\|T\|\cdot\|T^{-1}\|\le (1+12\e)D$.

The first step is to construct a mapping $F:\R^n\to Y$ that is a Lipschitz almost-extension of $f$, i.e., it is Lipschitz and on $\n_\d$ it takes values that are close to the corresponding values of $f$. The statement below is a refinement of a result of Bourgain~\cite{Bou87}. The proof of Bourgain's almost extension theorem has been significantly simplified by Begun~\cite{Beg99}, and our proof of Lemma~\ref{thm:ext1} below  follows Begun's argument; see Section~\ref{sec:ext proof}.

\begin{lemma}\label{thm:ext1} If $\delta<\frac{\e}{4n}$ then there exists a mapping $F:\R^n\to Y$ that is differentiable almost
everywhere on $\R^n$, is differentiable { everywhere} on $\frac{1}{2} B_X$,
and has the following properties.
\begin{itemize}
\item $F$ is supported on $3B_X$.
\item $\|F(x)-F(y)\|_Y\le 6\|x-y\|_X$ for all $x,y\in \R^n$.
\item $\|F(x)-F(y)\|_Y\le \left(1+\e\right)\|x-y\|_X$ for all $x,y\in \frac12 B_X$.
\item $\|F(x)-f(x)\|_Y\le \frac{9n\d}{\e}$ for all $x\in \n_\d$.
\end{itemize}
\end{lemma}

In what follows, the volume of a Lebesgue measurable set $A\subseteq \R^n$ will be denoted $\vol(A)$. For $t\in (0,\infty)$ the Poisson kernel $P_t:\R^n\to [0,\infty)$ is given by
$$
P_t(x) = \frac {c_nt} {\left(t^2+\|x\|_2^2\right)^{\frac{n+1}{2}}},
$$
where $c_n$ is the normalization factor ensuring that $\int_{\R^n} P_t(x)dx=1$.  Thus $c_n = \Gamma\left(\frac{n+1}2\right)/\pi^{\frac{n+1}2}$, as computed for example in~\cite[Sec.~X.3]{Tor86}. We will use repeatedly the standard semigroup property $P_t*P_s=P_{t+s}$, where as usual $f*g(x)=\int_{\R^n}f(y)g(y-x)dx$ for $f,g\in L_1(\R^n)$.

Assume from now on that $\delta<\frac{\e}{4n}$ and fix a mapping $F:\R^n\to Y$ satisfying the conclusion of Lemma~\ref{thm:ext1}. We will consider the  evolutes of $F$ under the Poisson semigroup, i.e., the functions $P_t*F: \R^n\to Y$ given by  $P_t*F(x)=\int_{\R^n}P_t(y-x)F(y)dy$. Our goal is to show that there exists $t_0\in (0,\infty)$ and $x\in \R^n$ such that the derivative $T=(P_{t_0}*F)'(x)$ is injective and satisfies $\|T\|\cdot\|T^{-1}\|\le (1+10\e)D$. Intuitively, one might expect this to happen for every small enough $t$, since in this case $P_t*F$ is close to $F$, and $F$ itself is close to a bi-Lipschitz map when restricted to the $\d$-net $\n_\d$. In reality, proving the existence of $t_0$ requires work; the existence of $t_0$ will be proved by contradiction, i.e., we will show that it cannot not exist, without pinpointing a concrete $t_0$ for which  $(P_{t_0}*F)'(x)$ has the desired properties.

\begin{lemma}\label{lem:contradiction}
Let $\mu$ be a Borel probability measure on $S_X$. Fix $R,A\in (0,\infty)$ and $m\in \N$. Then there exists $t\in (0,\infty)$ satisfying
\begin{equation}\label{eq:t range1}
\frac{A}{(R+1)^{m+1}}\le t\le A,
\end{equation}
such that
\begin{equation}\label{eq:stabilization}
\int_{S_X}\int_{\R^n} \|\partial_a(P_t*F)(x)\|_Ydxd\mu(a)\le \int_{S_X}\int_{\R^n} \|\partial_a(P_{(R+1)t}*F)(x)\|_Ydxd\mu(a)+\frac{6\vol(3B_X)}{m}.
\end{equation}
\end{lemma}

\begin{proof}
If~\eqref{eq:stabilization} fails for all $t$ satisfying~\eqref{eq:t range1} then for every $k\in \{0,\ldots,m+1\}$ we have
\begin{multline}\label{eq: to iterate}
\int_{S_X}\int_{\R^n} \left\|\partial_a\left(P_{A(R+1)^{k-m-1}}*F\right)(x)\right\|_Ydxd\mu(a)\\> \int_{S_X}\int_{\R^n} \left\|\partial_a\left(P_{A(R+1)^{k-m}}*F\right)(x)\right\|_Ydxd\mu(a)+\frac{6\vol(3B_X)}{m}.
\end{multline}
By iterating~\eqref{eq: to iterate} we get the estimate
\begin{multline}\label{eq:iterated}
\int_{S_X}\int_{\R^n} \left\|\partial_a\left(P_{A(R+1)^{-m-1}}*F\right)(x)\right\|_Ydxd\mu(a)\\> \int_{S_X}\int_{\R^n} \left\|\partial_a\left(P_{A(R+1)}*F\right)(x)\right\|_Ydxd\mu(a)+\frac{6(m+1)\vol(3B_X)}{m}.
\end{multline}

At the same time, since $F$ is differentiable almost everywhere and $6$-Lipschitz, for every $a\in S_X$ we have $\|\partial_aF\|_Y\le 6$ almost everywhere. Since $F$ is supported on $3B_X$, it follows that
\begin{multline}\label{eq:partial on ball}
\int_{\R^n} \left\|\partial_a\left(P_{A(R+1)^{-m-1}}*F\right)(x)\right\|_Ydx=\int_{\R^n} \left\|\left(P_{A(R+1)^{-m-1}}*\partial_a F\right)(x)\right\|_Ydx\\\le \int_{\R^n}\int_{\R^n}P_{A(R+1)^{-m-1}}(x-y)\|\partial_a F(y)\|_Y dxdy
=\int_{3B_X}\|\partial_aF(y)\|_Ydy \le 6\vol(3B_X).
\end{multline}
If we integrate~\eqref{eq:partial on ball} with respect to $\mu$, then since $\mu$ is a probability measure we obtain a contradiction to~\eqref{eq:iterated}
\end{proof}

In order to apply Lemma~\ref{lem:contradiction}, we will contrast it with the following key statement (proved in~Section~\ref{sec tech}), which asserts that the directional derivatives of $P_t*F$ are large after an appropriate averaging.
\begin{lemma}\label{lem:contrast}
Assume that $t\in (0,1/2]$, $R\in (0,\infty)$ and $\delta\in (0,\e/(4n))$ satisfy
\begin{equation}\label{eq:key assumption1}
\d\le \frac{\e t\log(7/t)}{2\sqrt{n}}\le \frac{\e^4}{6n^{5/2}(80D)^2},
\end{equation}
and
\begin{equation}\label{eq:def R}
\frac{720n^{3/2}D^2\log(7/t)}{\e^2}\le R\le \frac{\e}{32t\sqrt{n}}.
\end{equation}
Then for every $x\in \frac18 B_X$ and $a\in S_X$ we have
\begin{equation}\label{eq:directional}
\left(\|\partial_a(P_t*F)\|_Y*P_{Rt}\right)(x)\ge \frac{1-\e}{D}.
\end{equation}
\end{lemma}

We record one more (simpler) fact about the evolutes of $F$ under the Poisson semigroup.

\begin{lemma}\label{le:evolute Lip}
Assume that $0<t<\frac{\e}{25\sqrt{n}}$. Then for every $x,y\in \frac14 B_X$ we have
$$
\|P_t*F(x)-P_t*F(y)\|_Y\le (1+2\e)\|x-y\|_X.
$$
\end{lemma}

With the above tools at hand, we will now show how to conclude the proof of Theorem~\ref{thm:bourgain intro}. It will then remain to prove Lemma~\ref{thm:ext1} (in Section~\ref{sec:ext proof}), Lemma~\ref{lem:contrast} (Section~\ref{sec tech}) and Lemma~\ref{le:evolute Lip} (also in Section~\ref{sec tech}).

\begin{proof}[Proof of Theorem~\ref{thm:bourgain intro}] Assume that $\delta\in (0,1)$ satisfies
\begin{equation}\label{eq:our delta}
\d\le \left(\frac{\e}{cD}\right)^{12n(cD/\e)^{n+1}},
\end{equation}
where $c=300$ (this is an overestimate for the ensuing calculation). Fix  an $(\e/D)$-net $\mathcal F$ in $S_X$ with $|\mathcal F|\le (3D/\e)^n$ (for the existence of nets of this size, see e.g. \cite[Lem.~12.3.1]{Kal06}). Let $\mu$ be the uniform probability measure on $\mathcal F$. Define
\begin{equation}\label{eq:choices}
A=\left(\frac{\e}{cD}\right)^{5n},\quad R=\left(\frac{cD}{\e}\right)^{4n}-1,\quad m=\left\lfloor\left(\frac{cD}{\e}\right)^{n+1}\right\rfloor-1.
\end{equation}
Apply Lemma~\ref{lem:contradiction} with the above parameters, obtaining some $t\in (0,\infty)$ satisfying
\begin{equation}\label{eq:t range}
\left(\frac{\e}{cD}\right)^{12n\left(cD/\e\right)^{n+1}}\le t\le \left(\frac{\e}{cD}\right)^{5n},
\end{equation}
such that
\begin{equation}\label{eq:stabilization net}
\sum_{a\in \mathcal F}\int_{\R^n} \|\partial_a(P_t*F)(x)\|_Ydx\le \sum_{a\in \mathcal F}\int_{\R^n} \|\partial_a(P_{(R+1)t}*F)(x)\|_Ydx+\frac{6|\mathcal F|\vol(3B_X)}{m}.
\end{equation}
One checks that for $\delta$ satisfying~\eqref{eq:our delta}, $R$ as in~\eqref{eq:choices}, and any $t$ satisfying~\eqref{eq:t range}, inequalities~\eqref{eq:key assumption1} and~\eqref{eq:def R} are satisfied. Thus the conclusion~\eqref{eq:directional} of Lemma~\ref{lem:contrast} holds true for all $a\in S_X$ and $x\in \frac18 B_X$.

Note that by convexity we have for every $a\in S_X$ and almost every $x\in \R^n$,
$$
\|\partial_a(P_{(R+1)t}*F)(x)\|_Y= \|\left(P_{Rt}*\left(\partial_a(P_{t}*F)\right)\right)(x)\|_Y\le \left(\|\partial_a(P_t*F)\|_Y*P_{Rt}\right)(x).
$$
Thus $\|\partial_a(P_t*F)\|_Y*P_{Rt}-\|\partial_a(P_{(R+1)t}*F)\|_Y\ge 0$, so we may use Markov's inequality as follows.
\begin{eqnarray}\label{eq:for union}
&&\nonumber\!\!\!\!\!\!\!\!\!\!\!\!\!\!\!\!\!\!\!\!\!\!\!\!\!\!\!\!\!\!\!\!\vol\left(\left\{x\in \frac18B_X:\ \left(\|\partial_a(P_t*F)\|_Y*P_{Rt}\right)(x)-\|\partial_a(P_{(R+1)t}*F)(x)\|_Y\ge \frac{\e}{D}\right\}\right)\\\nonumber&\le&
\frac{D}{\e}\left(\int_{\R^n}\left(\left(\|\partial_a(P_t*F)\|_Y*P_{Rt}\right)(x)-\|\partial_a(P_{(R+1)t}*F)(x)\|_Y\right)dx\right)\\&=&\frac{D}{\e}\left(\int_{\R^n} \|\partial_a(P_t*F)(x)\|_Ydx-
\int_{\R^n} \|\partial_a(P_{(R+1)t}*F)(x)\|_Ydx\right).
\end{eqnarray}
Hence,
\begin{eqnarray*}
&&\nonumber\!\!\!\!\!\!\!\!\!\!\!\!\!\!\!\!\!\!\!\!\!\!\!\vol\left(\left\{x\in \frac18B_X:\ \exists a\in \mathcal F,\ \ \ \left(\|\partial_a(P_t*F)\|_Y*P_{Rt}\right)(x)-\|\partial_a(P_{(R+1)t}*F)(x)\|_Y\ge \frac{\e}{D}\right\}\right)\\
&\stackrel{\eqref{eq:for union}}{\le}& \frac{D}{\e}\left(\sum_{a\in \mathcal F}\int_{\R^n} \|\partial_a(P_t*F)(x)\|_Ydx-\sum_{a\in \mathcal F}\int_{\R^n} \|\partial_a(P_{(R+1)t}*F)(x)\|_Ydx\right)\\
&\stackrel{\eqref{eq:stabilization net}}{\le}& \frac{D}{\e}\cdot\frac{6|\mathcal F|\vol(3B_X)}{m}\\
&\stackrel{\eqref{eq:choices}}{\le}& \frac{12D}{\e}\left(\frac{3D}{\e}\right)^n\left(\frac{\e}{cD}\right)^{n+1}(24)^n\vol\left(\frac18 B_X\right)\\
&=&\frac{6^n}{25^{n+1}}\vol\left(\frac18 B_X\right)<\vol\left(\frac18 B_X\right).
\end{eqnarray*}
Consequently, there exists $x\in \frac18B_X$ satisfying
\begin{equation}\label{eq:use 12}
\forall a\in \mathcal F,\quad\left(\|\partial_a(P_t*F)\|_Y*P_{Rt}\right)(x)-\|\partial_a(P_{(R+1)t}*F)(x)\|_Y< \frac{\e}{D}.
 \end{equation}
 But we already argued that~\eqref{eq:directional} holds as well, so~\eqref{eq:use 12} implies that
\begin{equation}\label{eq:at a}
\forall a\in \mathcal F,\quad \|\partial_a(P_{(R+1)t}*F)(x)\|_Y\ge \frac{1-2\e}{D}.
\end{equation}
 Note that by~\eqref{eq:choices} and~\eqref{eq:t range} we have $(R+1)t\le (\e/(cD))^n<\e/\left(25\sqrt{n}\right)$. Hence, if we define $T=(P_{(R+1)t}*F)'(x)$ then by Lemma~\ref{le:evolute Lip} we have $\|T\|\le 1+2\e$.  By~\eqref{eq:at a}, $\|Ta\|_Y\ge (1-2\e)/D$ for all $a\in \mathcal F$. For $z\in S_X$ take $a\in \mathcal F$ such that $\|z-a\|_X\le \e/D$. Then,
 $$
 \|Tz\|\ge \|Ta\|-\|T\|\cdot\|z-a\|_X\ge \frac{1-2\e}{D}-(1+2\e)\frac{\e}{D}\ge \frac{1-4\e}{D}.
 $$
Hence $T$ is invertible and $\|T^{-1}\|\le D/(1-4\e)$. Thus $\|T\|\cdot\|T^{-1}\|\le \frac{1+2\e}{1-4\e}D\le(1+12\e)D$.
\end{proof}

\section{Proof of Lemma~\ref{thm:ext1}}\label{sec:ext proof}

We will use the following lemma of Begun~\cite{Beg99}.
\begin{lemma}\label{lem:begun}
Let $K\subseteq \R^n$ be a convex set and fix $\tau,\eta,L\in (0,\infty)$. Assume that we are given a mapping $h: K+\tau B_X\to Y$ satisfying $\|h(x)-h(y)\|_Y\le L\left(\|x-y\|_X+\eta\right)$ for all $x,y\in K+\tau B_X$. Define $H:K\to Y$ by
$$
H(x)=\frac{1}{\tau^n\vol(B_X)}\int_{\tau B_X}h(x-y)dy.
$$
Then $\|H(x)-H(y)\|_Y\le L\left(1+\frac{n\eta}{2\tau}\right)\|x-y\|_X$ for all $x,y\in K$.
\end{lemma}
We refer to~\cite{Beg99} for an elegant proof of Lemma~\ref{lem:begun}. The deduction of Lemma~\ref{thm:ext1}  from Lemma~\ref{lem:begun} is via the following simple partition of unity argument. Let $\{\f_p:\R^n\to [0,1]\}_{p\in \n_\d}$ be a family of smooth functions satisfying $\sum_{p\in \n_\d}\f_p(x)=1$ for all $x\in B_X$ and $\f_p(x)=0$ for all $(p,x)\in \n_\d\times \R^n$ with $\|x-p\|_X\ge 2\delta$. A standard construction of such functions can be obtained by taking a smooth $\psi:\R^n\to [0,1]$ which is equals $1$ on $B_X$ and vanishes outside $2B_X$, and defining $\psi_p(x)=\psi((x-p)/\delta)$ for $(p,x)\in \n_\d\times \R^n$. If we then write $\n_\d=\{p_1,p_2,\ldots,p_N\}$, define $\f_{p_1}=\psi_{p_1}$ and $\f_{p_j}=\psi_{p_j}\prod_{i=1}^{j-1}(1-\psi_{p_i})$ for $j\in \{2,\ldots,N\}$. Then $\sum_{p\in \n_\d}\f_p=1-\prod_{p\in \n_\d}(1-\psi_p)=1$ on $B_X$ since every $x\in B_X$ satisfies $\|x-p\|_X\le \d$ for some $p\in \n_\d$.

Now define $g:B_X\to Y$ by $g(x)=\sum_{p\in \n_\d} \f_p(x)f(p)$. Setting $\beta(t)=\max\{0,2-t\}$ for $t\in [0,\infty)$, consider the mapping $h:\R^n\to Y$ given by
\begin{equation}\label{eq:defh}
h(x)=\left\{\begin{array}{ll}g(x)&\mathrm{if\ }x\in B_X,\\
\beta(\|x\|_X)g\left(x/\|x\|_X\right)&\mathrm{if\ }x\in \R^n\setminus B_X.\end{array}\right.
\end{equation}
 Observe that if $x,y\in B_X$ then
\begin{multline*}
h(x)-h(y)=g(x)-g(y)=\sum_{p\in \n_\d\cap (x+2\delta B_X)}\f_p(x)f(p)-\sum_{q\in \n_\d\cap (y+2\delta B_X)}\f_q(y)f(q)\\
= \sum_{\substack{p\in \n_\d\cap (x+2\delta B_X)\\q\in \n_\d\cap (y+2\delta B_X)}}\f_p(x)\f_q(y)\left[f(p)-f(q)\right].
\end{multline*}
This identity implies that
\begin{equation}\label{eq:on ball}
\forall x,y\in B_X,\quad \|h(x)-h(y)\|_Y\le \|x-y\|_X+4\delta
\end{equation}
If $x\in B_X$ and $y\in \R^n\setminus B_X$ then using $f(\n_\d)\subseteq 2B_Y$ and the fact that $\beta$ is $1$-Lipschitz,
\begin{multline*}
\|h(x)-h(y)\|_Y\le \left\|g(x)-g\left(\frac{y}{\|y\|_X}\right)\right\|_Y+\left(1-\beta(\|y\|_X)\right)\left\|g\left(\frac{y}{\|y\|_X}\right)\right\|_Y\\\stackrel{\eqref{eq:on ball}}{\le} \left\|x-\frac{y}{\|y\|_X}\right\|_X+4\delta+\left(\|y\|_X-1\right)\sup_{p\in \n_\d}\|f(p)\|_Y\le \|x-y\|_X+3(\|y\|_X-1)+4\delta.
\end{multline*}
Since $\|y\|_X-1\le \|x-y\|_X+\|x\|_X-1\le \|x-y\|_X$, it follows that
\begin{equation}\label{eq:on and off ball}
\forall x\in B_X,\forall y\in \R^n\setminus B_X,\quad \|h(x)-h(y)\|_Y\le 4(\|x-y\|_X+\delta).
\end{equation}
If $x,y\in \R^n\setminus B_X$ then
\begin{multline}\label{eq:off ball}
\|h(x)-h(y)\|_Y\le \left\|g\left(\frac{x}{\|x\|_X}\right)-g\left(\frac{y}{\|y\|_X}\right)\right\|_Y\beta(\|x\|)+
\left\|g\left(\frac{y}{\|y\|_X}\right)\right\|_Y\left|\beta(\|x\|_X)-\beta(\|y\|_X)\right|\\
\stackrel{\eqref{eq:on ball}}{\le} \left\|\frac{x}{\|x\|_X}-\frac{y}{\|y\|_X}\right\|_X+4\d+2\left\|x-y\right\|_X\le 4(\|x-y\|_X+\delta).
\end{multline}

Set $\tau=2n\delta/\e\in (0,1/2)$  and define for $x\in \R^n$,
\begin{equation}\label{eq:defF}
F(x)=\frac{1}{\tau^n\vol(B_X)}\int_{\tau B_X}h(x-y)dy.
\end{equation}
It follows from the definition~\eqref{eq:defh} that $h$ is differentiable almost everywhere on $\R^n$. Since $h$ is differentiable on $B_X\setminus S_X$ and $\tau\in (0,1/2)$, it follows from~\eqref{eq:defF} that $F$ is differentiable almost everywhere on $\R^n$, and is differentiable everywhere on $\frac12 B_X$. Clearly $F$ is supported on $(2+\tau)B_X\subseteq 3B_X$, i.e., the first assertion of Lemma~\ref{thm:ext1} holds. Due to~\eqref{eq:on ball}, \eqref{eq:on and off ball}, \eqref{eq:off ball}, an application of Lemma~\ref{lem:begun} with $K=\R^n$, $L=4$ and $\eta=\delta$ shows that $F$ is $4\left(1+\e/2\right)$-Lipschitz on $\R^n$, proving the second assertion of Lemma~\ref{thm:ext1}. Due to~\eqref{eq:on ball}, an application of Lemma~\ref{lem:begun} with $K=(1-\tau)B_X$ shows that $F$ is $\left(1+\e\right)$-Lipschitz on $(1-\tau)B_X\supseteq \frac12 B_X$. This establishes the third assertion of Lemma~\ref{thm:ext1}.
To prove the fourth assertion of Lemma~\ref{thm:ext1}, fix $x\in \n_\d$. Then,
\begin{equation}\label{eq:approx first}
\left\|F(x)-h(x)\right\|_Y\le \frac{1}{\tau^n\vol(B_X)}\int_{\tau B_X}\left\|h(x-y)-h(x)\right\|_Ydy\stackrel{\eqref{eq:on ball}\wedge\eqref{eq:on and off ball}}{\le} 4(\tau+\delta).
\end{equation}
Also,
\begin{equation}\label{eq:approx secon}
\|h(x)-f(x)\|_Y\le \sum_{p\in \n_\d} \|f(x)-f(p)\|_Y\f_p(x)\le \max_{p\in \n_\delta\cap(x+2\d B_X)} \|f(x)-f(p)\|_Y\le 2\d.
\end{equation}
Recalling that $\tau=2n\d/\e$, the fourth assertion on Lemma~\ref{thm:ext1} follows from~\eqref{eq:approx first} and~\eqref{eq:approx secon}.\qed

\section{Proof of Lemma~\ref{le:evolute Lip} and Lemma~\ref{lem:contrast}}\label{sec tech}

 We will need the following standard estimate, which holds for all $r,t\in (0,\infty)$.
\begin{equation}\label{eq:tail}
\int_{\R^n\setminus\left(rB_X\right)}P_t(x)dx\le \frac{t\sqrt{n}}{r}.
\end{equation}
To check~\eqref{eq:tail}, letting $s_{n-1}$ denote the surface area of the unit Euclidean sphere $S^{n-1}$, and recalling that $P_t(x)=t^{-n}P_1(x/t)$, we have
\begin{multline*}
\int_{\|x\|_X\ge r} P_t(x)dx \stackrel{\eqref{eq:john assumption}}{\le} \int_{\|x\|_2\ge r} P_t(x)dx=\int_{\|x\|_2\ge r/t} P_1(x)dx\\=c_ns_{n-1}\int_{r/t}^\infty\frac{s^{n-1}}{(1+s^2)^{\frac{n+1}{2}}}ds
\le c_ns_{n-1}\int_{r/t}^\infty \frac{ds}{s^2}=\frac{c_ns_{n-1}t}{r}.
\end{multline*}
It remains to recall that $c_n = \Gamma\left(\frac{n+1}2\right)/\pi^{\frac{n+1}2}$ and $
s_{n-1}=n\pi^{\frac{n}{2}}/\Gamma\left(\frac{n}{2}+1\right)$ (see e.g. \cite[Sec.~1]{Ball97}), and, using Stirling's formula, to obtain the estimate $c_ns_{n-1}\le \sqrt{2n/\pi}$.

Another standard estimate that we will use is that for every $y\in \R^n$ we have
\begin{equation}\label{eq:Lipschitz Poisson}
\int_{\R^n} \left|P_t(x)-P_t(x+y)\right|dx\le \frac{\sqrt{n}\|y\|_2}{t}.
\end{equation}
Since $P_t(x)=t^{-n}P_1(x/t)$ it suffices to check~\eqref{eq:Lipschitz Poisson} when $t=1$. Now,
\begin{multline*}
\int_{\R^n} \left|P_1(x)-P_1(x+y)\right|dx=\int_{\R^n}\left|\int_0^1\langle \nabla P_1(x+sy),y\rangle ds\right|dx\le \|y\|_2\int_{\R^n}\left\|\nabla P_1(x)\right\|_2dx\\
=(n+1)c_n\|y\|_2\int_{\R^n}\frac{\|x\|_2}{\left(1+\|x\|_2^2\right)^{\frac{n+3}{2}}}dx=(n+1)c_ns_{n-1}
\|y\|_2\int_0^\infty \frac{r^n}{(1+r^2)^{\frac{n+3}{2}}}dr=c_ns_{n-1}\|y\|_2,
\end{multline*}
where we used the fact that the derivative of $r^{n+1}/(1+r^2)^{\frac{n+1}{2}}$ equals $(n+1)r^n/(1+r^2)^{\frac{n+3}{2}}$. The required estimate~\eqref{eq:Lipschitz Poisson} now follows from Stirling's formula.

\begin{proof}[Proof of Lemma~\ref{le:evolute Lip}] We have,
\begin{eqnarray*}
\|P_t*F(x)-P_t*F(y)\|_Y&\le& \int_{\R^n} P_t(z)\|F(x-z)-F(y-z)\|_Ydz\\
&\stackrel{(*)}{\le}& \left((1+\e)\int_{\frac14 B_X}P_t(z)dz+6\int_{\R^n\setminus \left(\frac14 B_X\right)} P_t(z)dz\right)\|x-y\|_X\\&\stackrel{\eqref{eq:tail}}{\le}& \left(1+\e+24t\sqrt{n}\right)\|x-y\|_X,
\end{eqnarray*}
where in $(*)$ we used the fact that $F$ is $(1+\e)$-Lipschitz on $\frac12 B_X$ and $6$-Lipschitz on $\R^n$.
\end{proof}

\begin{lemma}\label{lem:evolute close}
For every $t\in (0,1/2]$ and every $x\in B_X$ we have
$$
\left\|P_t*F(x)-F(x)\right\|_Y\le 8\sqrt{n}t\log\left(\frac{7}{t}\right).
$$
\end{lemma}

\begin{proof}
Since $F$ is supported on $3B_X$,
\begin{multline}\label{eq:perturn 2 term}
\left\|P_t*F(x)-F(x)\right\|_Y\\\le \int_{x+3B_X}\|F(y-x)-F(x)\|_YP_t(y)dy+\|F(x)\|_Y\int_{y\in \R^n\setminus (x+3B_X)}P_t(y)dy.
\end{multline}
Since $F$ is $6$-Lipschitz and it vanishes outside $3B_X$, we have $\|F(x)\|_Y\le 18$. Moreover, if $\|y-x\|_X\ge3$ then $\|y\|_X\ge \|x-y\|_X-\|x\|_X\ge 2$, and therefore
\begin{equation}\label{eq:perturn second term}
\|F(x)\|_Y\int_{y\in \R^n\setminus (x+3B_X)}P_t(y)dy\le 18\int_{\R^n\setminus(2B_X)} P_t(y)dy \stackrel{\eqref{eq:tail}}{\le}9t\sqrt{n}.
\end{equation}
To bound the first term in the right hand side of~\eqref{eq:perturn 2 term} note that if $\|y-x\|_X\le 3$ then $\|y\|_2\le \sqrt{n}\|y\|_X\le 4\sqrt{n}$. Moreover, $\|F(y-x)-F(x)\|_X\le 6\|y\|_X\le 6\|y\|_2$. Hence,
\begin{multline}\label{eq:before opt s}
\int_{x+3B_X}\|F(y-x)-F(x)\|_YP_t(y)dy\le 6\int_{\|y\|_2\le 4\sqrt{n}} \|y\|_2P_t(y)dy\\=6t\int_{\|y\|_2\le \frac{4\sqrt{n}}{t}} \|y\|_2P_1(y)dy=
6tc_ns_{n-1}\int_0^{\frac{4\sqrt{n}}{t}}\frac{s^n}{(1+s^2)^{\frac{n+1}{2}}}ds
\end{multline}
Direct differentiation shows that the maximum of $s^n/(1+s^2)^{\frac{n+1}{2}}$ is attained when $s=\sqrt{n}$, and therefore $s^n/(1+s^2)^{\frac{n+1}{2}}\le \min\{1/\sqrt{en},1/s\}$ for all $s\in (0,\infty)$. Hence,
\begin{equation}\label{eq:log}
\int_0^{\frac{4\sqrt{n}}{t}}\frac{s^n}{(1+s^2)^{\frac{n+1}{2}}}ds\le 1+\int_{\sqrt{en}}^{\frac{4\sqrt{n}}{t}}\frac{ds}{s}=1+\log\left(\frac{4}{t\sqrt{e}}\right).
\end{equation}
The required result now follows from substituting~\eqref{eq:perturn second term}, \eqref{eq:before opt s}, \eqref{eq:log} into~\eqref{eq:perturn 2 term}, and using the fact that $t\le 1/2$ and $c_ns_{n-1}\le \sqrt{n}$.
\end{proof}

\begin{proof}[Proof of Lemma~\ref{lem:contrast}] Define
\begin{equation}\label{eq:def Theta}
\Theta= \frac{100D\sqrt{n}t\log(7/t)}{\e}.
\end{equation}
For $w,y\in \frac12 B_X$ let $p,q\in \n_\d\cap(\frac12 B_X)$ satisfy $\|p-w\|_X,\|q-y\|_Y\le 2\d$. Assume that $\|w-y\|_X\ge \Theta$. Using the third and fourth assertions of Lemma~\ref{thm:ext1}, together with Lemma~\ref{lem:evolute close}, we have
\begin{eqnarray}\label{eq:lower F on net}
&&\nonumber\!\!\!\!\!\!\!\!\!\!\!\!\left\|(P_t*F)(w)-(P_t*F)(y)\right\|_Y\ge \|f(p)-f(q)\|_Y-\|F(p)-f(p)\|_Y-\|F(q)-f(q)\|_Y\\&&\nonumber\!\!\!\!\!\!-\|F(w)-F(p)\|_Y-
\|F(y)-F(q)\|_Y-\|(P_t*F)(w)-F(w)\|_Y-\|(P_t*F)(y)-F(y)\|_Y\\
&\stackrel{\eqref{eq:assumption on net}}{\ge}& \nonumber \frac{\|p-q\|_Y}{D}-\frac{18n\d}{\e}-4(1+\e)\d-16\sqrt{n}t\log\left(\frac{7}{t}\right)\\
&\ge&\nonumber \frac{\|w-y\|_X-4\d}{D}-\frac{18n\d}{\e}-4(1+\e)\d-16\sqrt{n}t\log\left(\frac{7}{t}\right)\\&\ge& \frac{1-\e/3}{D}\|w-y\|_X,
\end{eqnarray}
where~\eqref{eq:lower F on net} uses the assumptions $\|w-y\|_X\ge \Theta$ and~\eqref{eq:key assumption1}.

Note that the second inequality in~\eqref{eq:key assumption1} implies that $\Theta\le 1/4$. Therefore, since $\|a\|_X=1$ it follows from~\eqref{eq:lower F on net} that for every $z\in \frac14 B_X$,
\begin{multline}\label{eq:integrated}
\frac{1-\e/3}{D}\Theta\le \|P_t*F(z+\Theta a)-P_t*F(z)\|_Y\\=\left\|\int_0^\Theta \partial_a(P_t*F)(z+sa)ds\right\|_Y
\le \int_0^\Theta\left\| \partial_a(P_t*F)(z+sa)\right\|_Yds.
\end{multline}
Since in the statement of Lemma~\ref{lem:contrast} we are assuming that $\|x\|_X\le 1/8$,
\begin{multline}\label{eq:smoothed convolution}
\frac{1}{\Theta}\int_0^\Theta\int_{\R^n} \|\partial_a(P_t*F)(x+sa-y)\|_YP_{Rt}(y)dyds\stackrel{\eqref{eq:integrated}}{\ge} \frac{1-\e/3}{D}\int_{\frac18 B_X} P_{Rt}(y)dy\\
\stackrel{\eqref{eq:tail}}{\ge} \frac{1-\e/3}{D}\left(1-8Rt\sqrt{n}\right)\stackrel{\eqref{eq:def R}}{\ge}\frac{(1-\e/3)(1-\e/4)}{D}\ge \frac{1-7\e/12}{D}.
\end{multline}
Since $F$ is $6$-Lipschitz,  $\|\partial_a F\|_Y\le 6$ almost everywhere, and therefore $\|\partial_a(P_t*F)\|_Y\le 6$ almost everywhere. Hence,
\begin{multline}\label{eq:diff from smoothed}
\int_0^\Theta\int_{\R^n} \|\partial_a(P_t*F)(x-y)\|_Y\left(P_{Rt}(y+sa)-P_{Rt}(y)\right)dy ds\\\le 6\int_0^\Theta \int_{\R^n}\left|P_{Rt}(y+sa)-P_{Rt}(y)\right|dyds
\stackrel{\eqref{eq:Lipschitz Poisson}}{\le} \frac{6\sqrt{n}\|a\|_2}{Rt}\cdot \frac{\Theta^2}{2}\stackrel{\eqref{eq:john assumption}}{\le}\frac{3n\Theta^2}{Rt}
\stackrel{\eqref{eq:def R}\wedge\eqref{eq:def Theta}}{\le}\frac{5\e\Theta}{12D}.
\end{multline}
We can now conclude the proof of Lemma~\ref{lem:contrast} as follows.
\begin{multline*}
\left(\|\partial_a(P_t*F)\|_Y*P_{Rt}\right)(x)=\frac{1}{\Theta}\int_0^\Theta\int_{\R^n} \|\partial_a(P_t*F)(x+sa-y)\|_YP_{Rt}(y)dyds\\-\frac{1}{\Theta}\int_0^\Theta\int_{\R^n} \|\partial_a(P_t*F)(x-y)\|_Y\left(P_{Rt}(y+sa)-P_{Rt}(y)\right)dy ds
 \stackrel{\eqref{eq:smoothed convolution}\wedge\eqref{eq:diff from smoothed}}{\ge} \frac{1-\e}{D}.\qedhere
\end{multline*}
\end{proof}

\section{Proof of Theorem~\ref{thm:with Gid}}\label{sec:use JMS}

The following general lemma will be used later; compare to~\cite[Prop.~1]{JMS09}.

\begin{lemma}\label{lem:U}
Let $(V,\|\cdot\|_V)$ be a Banach space and $U=(\R^n,\|\cdot\|_U)$ be an $n$-dimensional Banach space. Assume that  $g: B_U\to V$ is continuous and everywhere differentiable on the interior of $B_U$. Then
\begin{equation}\label{eq:div}
\left\|\frac{1}{\vol(B_U)}\int_{B_U}g'(u)du\right\|_{U\to V}\le n\|g\|_{L_\infty(S_U)}.
\end{equation}
\end{lemma}

\begin{proof}
Fix $y\in \R^n$ with $\|y\|_2=1$. Let $P_{y^\perp}:\R^n\to y^\perp$ denote the orthogonal projection onto the hyperplane $y^\perp$. For every $u\in P_{y^\perp}(B_U)$ there are unique $a_u,b_u\in \R$ satisfying $a_u\le b_u$ and $\|u+a_u y\|_U=\|u+b_u y\|_U=1$. Hence,
\begin{multline*}
\left\|\frac{1}{\vol(B_U)}\int_{B_U}g'(u)(y)du\right\|_{V}=\left\|\frac{1}{\vol(B_U)}\int_{P_{y^{\perp}}(B_U)} \int_{a_u}^{b_u} \frac{d}{ds} g(u+sy) ds du\right\|_V\\=
\left\|\frac{1}{\vol(B_U)}\int_{P_{y^\perp}(B_U)} \left(g(u+b_u y) -g(u+a_u y)\right) du\right\|_V\le 2\|g\|_{L_\infty(S_U)}\cdot\frac{\vol_{n-1} (P_{y^\perp}(B_U))}{\vol(B_U)}.
\end{multline*}
Therefore, in order to prove~\eqref{eq:div} it suffices to show that $\vol_{n-1} (P_{y^\perp}(B_U))\le n\|y\|_U\vol(B_U)/2$. This is the same as  $\vol(K)\le \vol(B_U)$, where $K$ is the convex hull of $P_{y^\perp}(B_U)\cup\{\pm y/\|y\|_U\}$, i.e., $K$ is the union of the two cones with base $P_{y^\perp}(B_U)$ and cusp $\pm y/\|y\|_U$. For $u\in P_{y^{\perp}}(B_U)$ let $c_u\in [1,\infty)$ be the largest $c\in [1,\infty)$ for which $cu\in P_{y^{\perp}}(B_U)$. Then
\begin{equation}\label{eq:intervals in double cone}
K=\bigcup_{u\in P_{y^{\perp}}(B_U)} \left(u+\left[-\frac{c_u-1}{c_u\|y\|_U},\frac{c_u-1}{c_u\|y\|_U}\right]y\right).
\end{equation}
Recalling the definition of $a_u$ above, by the definition of $c_u$ we have $c_uu+a_{c_uu}y\in S_U$. Hence, since $\pm y/\|y\|_U\in B_U$, by convexity we know that the points
$$
\frac{1}{c_u}\left(c_uu+a_{c_uu}y\right)+\left(1-\frac{1}{c_u}\right)\frac{y}{\|y\|_U}\qquad \mathrm{and}\qquad \frac{1}{c_u}\left(c_uu+a_{c_uu}y\right)-\left(1-\frac{1}{c_u}\right)\frac{y}{\|y\|_U}
$$
are both in $B_U$. Consequently, by convexity again, we have that
\begin{equation}\label{eq:interval in BU}
\bigcup_{u\in P_{y^{\perp}}(B_U)} \left( u+\left[\frac{a_{c_uu}}{c_u}-\frac{c_u-1}{c_u\|y\|_U},\frac{a_{c_uu}}{c_u}+\frac{c_u-1}{c_u\|y\|_U}\right]y\right)\subseteq B_U.
\end{equation}
Since, by Fubini, the volume of the left hand side of~\eqref{eq:interval in BU} equals the volume of the right hand side of~\eqref{eq:intervals in double cone}, we conclude the desired estimate  $\vol(K)\le \vol(B_U)$.
\end{proof}

Fix $\e,\d\in (0,1/2)$ and let $\mathcal{N}_\d$ be a $\d$-net  in $B_X\subseteq \R^n$. Fixing also $D\in (1,\infty)$, assume that $f:\mathcal{N}_\d\to Y$ satisfies $\|x-y\|_X/D\le \|f(x)-f(y)\|_Y\le \|x-y\|_X$ for all $x,y\in \mathcal{N}_\d$. Define $Z=\mathrm{span}\left(f(\mathcal{N}_\d)\right)$. Thus $Z$ is a finite dimensional subspace of $Y$. Assume that
\begin{equation}\label{eq:delta assumption L_p(Z)}
\d\le \frac{\e^2}{30n^2D}.
\end{equation}
Since consequently $\d<\e/(4n)$, there exists $F:X\to Z$ that is differentiable everywhere on $\frac12B_X$ and satisfies the conclusion of Lemma~\ref{thm:ext1}. Let $\nu$ be the normalized Lebesgue measure on $\frac12 B_X$ and define a linear operator $T:X\to L_\infty(\nu,Z)$ by
\begin{equation}\label{eq:def T}
(Ty)(x)=F'(x)(y).
\end{equation}
Since $F$ is $(1+\e)$-Lipschitz on $\frac12 B_X$ we have the operator norm bound
\begin{equation*}\label{eq:T norm}
\|T\|_{X\to L_\infty(\nu,Z)}\le 1+\e.
\end{equation*}
Theorem~\ref{thm:with Gid} will therefore be proven once we show that for all $y\in X$ we have
\begin{equation}\label{eq:T goal}
\frac{1-\e}{D}\|y\|_X\le \|T y\|_{L_1(\nu,Z)}= \frac{1}{\vol\left(\frac12 B_X\right)}\int_{\frac12 B_X}\left\|F'(x)(y)\right\|_Ydx.
\end{equation}

To prove~\eqref{eq:T goal}, let $J:X\to \ell_\infty$ be a linear isometric embedding. By the nonlinear Hahn-Banach theorem (see e.g.~\cite[Ch.~1]{BL00}) there exists a mapping $G:Z\to\ell_\infty$ satisfying
\begin{equation}\label{eq:def G}
\forall x\in \mathcal{N}_\d,\quad G(f(x))=J(x)
\end{equation}
and $G$ is $D$-Lipschitz; we are extending here the mapping $J\circ \left(f^{-1}|_{f(\mathcal{N}_\d)}\right):f(\mathcal{N_\d})\to \ell_\infty$ while preserving its Lipschitz constant. By convolving $G$ with a smooth bump function whose integral on $Y$ equals $1$ and whose support has a small diameter, we can find $H: Z\to\ell_\infty$ with Lipschitz constant at most $D$ and satisfying
\begin{equation}\label{eq:def H}
\forall z\in F(B_X),\quad\|H(z)-G(z)\|_{\ell_\infty}\le \frac{nD\d}{\e}.
\end{equation}
Define a linear operator $S:L_1(\nu,Z)\to \ell_\infty$ by setting for $h\in L_1(\nu,Z)$,
\begin{equation}\label{eq:def S}
Sh=\int_{\frac12 B_X} H'(F(x))(h(x))d\nu(x).
\end{equation}
Since  $H$ is $D$-Lipschitz  and $\nu$ is a probability measure, we have the operator norm bound
\begin{equation}\label{eq:S norm}
\|S\|_{L_1(\nu,Z)\to \ell_\infty}\le D.
\end{equation}
By the chain rule, for every $y\in X$ we have
\begin{equation}\label{eq:chain}
ST(y)\stackrel{\eqref{eq:def T}\wedge\eqref{eq:def S}}{=}\int_{\frac12 B_X} H'(F(x))(F'(x)(y))d\nu(x)= \int_{\frac12 B_X} \left(H\circ F\right)'(x)(y)d\nu(x).
\end{equation}
Note that if $y\in \mathcal{N}_\d $ then
\begin{eqnarray}\label{eq:penultimate}
\nonumber\|H(F(y))-Jy\|_{\ell_\infty}&\stackrel{\eqref{eq:def G}}{=}&\|H(F(y))-G(f(y))\|_{\ell_\infty}\\&\le& \nonumber\|H(F(y))-G(F(y))\|_{\ell_\infty}+\|G(F(y))-G(f(y))\|_{\ell_\infty}
\\\nonumber&\stackrel{\eqref{eq:def H}}{\le}& \frac{nD\d}{\e}+D\|F(y)-f(y)\|_Y\\&\le&\frac{nD\d}{\e}+D\cdot \frac{9n\d}{\e}\le \frac{10nD\d}{\e},
\end{eqnarray}
where in the penultimate inequality in~\eqref{eq:penultimate} we used the fact that  $\|F(y)-f(y)\|_Y\le 9n\d/\e$ for all $y\in \mathcal N_\d$, due to Lemma~\ref{thm:ext1}.
If $x\in \frac12 B_X$ then there exists $y\in \mathcal{N}_\d\cap\left(\frac12 B_X\right)$ satisfying $\|x-y\|_X\le 2\d$. Using the fact that $H\circ F$ is $(1+\e)D$-Lipschitz on $\frac12 B_X$, it follows that
\begin{multline}\label{eq:reason for n^2}
\|H(F(x))-Jx\|_{\ell_\infty}\le \|H(F(y))-Jy\|_{\ell_\infty}+\|H(F(x))-H(F(y))\|_{\ell_\infty}+\|Jx-Jy\|_{\ell_\infty}\\\stackrel{\eqref{eq:penultimate}}{\le}
 \frac{10nD\d}{\e}+(1+\e)D\cdot 2\d+2\d\le \frac{15nD\d}{\e}.
\end{multline}
By Lemma~\ref{lem:U} with $V=\ell_\infty$, $\|\cdot\|_U=2\|\cdot\|_X$ and $g=H\circ F-J$, it follows from~\eqref{eq:reason for n^2} that
\begin{equation}\label{eq:use convex body lemma}
\|ST-J\|_{X\to\ell_\infty}\stackrel{\eqref{eq:chain}}{=}\left\|\int_{\frac12 B_X} (H\circ F)'(x)d\nu(x)-J\right\|_{X\to\ell_\infty}\le \frac{30n^2D\d}{\e}\stackrel{\eqref{eq:delta assumption L_p(Z)}}{\le}\e.
\end{equation}
It follows that for all $y\in X$,
$$
D\|T y\|_{L_1(\nu,Z)}\stackrel{\eqref{eq:S norm}}{\ge} \|STy\|_{\ell_\infty}\ge \|Jy\|_{\ell_\infty}-\|ST-J\|_{X\to\ell_\infty}\cdot\|y\|_X\stackrel{\eqref{eq:use convex body lemma}}{\ge} (1-\e)\|y\|_X.
$$
This concludes the proof of~\eqref{eq:T goal}, and hence the proof of Theorem~\ref{thm:with Gid} is complete.\qed

\bibliographystyle{abbrv}
\bibliography{bourgain-discrete}
\end{document}